\documentclass[preprint,12pt]{elsarticle}

\usepackage{amssymb}
\usepackage{amsthm}

\newcommand{\sysn}{\left\{\begin{array}{rcl}}
\newcommand{\sysk}{\end{array}\right.}

\newtheorem{theorem}{Theorem}[section]

\theoremstyle{example}

\newtheorem{proposition}[theorem]{Proposition}
\theoremstyle{definition}

\newtheorem{corollary}[theorem]{Corollary}

\begin{document}

\begin{frontmatter}



\title{ The application of selection principles in the study of the properties of function spaces  \tnoteref{label1}}


\author{Alexander V. Osipov}

\ead{OAB@list.ru}


\address{Krasovskii Institute of Mathematics and Mechanics, Ural Federal
 University,

 Ural State University of Economics, Yekaterinburg, Russia}

\begin{abstract} For a Tychonoff space $X$, we denote by $C_p(X)$
the space of all real-valued continuous functions on $X$ with the
topology of pointwise convergence.

In this paper we prove that:

$\bullet$ if every finite power of $X$ is Lindel$\ddot{e}$of then
$C_p(X)$ to be strongly sequentially separable iff  $X$ is
$\gamma$-set.

$\bullet$  $B_{\alpha}(X)$ --- functions Baire class $\alpha$
($1<\alpha\leq\omega_1$) on a Tychonoff space $X$ with the
pointwise topology --- to be sequentially separable iff there
 exists a Baire isomorphism class $\alpha$ from a space $X$ onto a
 $\sigma$-set.

$\bullet$ $B_{\alpha}(X)$ is strongly sequentially separable iff
$iw(X)=\aleph_0$ and $X$ is a $Z^{\alpha}$-cover $\gamma$-set for
$0<\alpha\leq \omega_1$.

$\bullet$ there is a consistent example of a set of reals $X$,
such that $C_p(X)$ is strongly sequentially separable, but
$B_1(X)$ is not strongly sequentially separable.

$\bullet$  $B(X)$ is sequentially separable, but is not strongly
sequentially separable for a $\mathfrak{b}$-Sierpi$\acute{n}$ski
set $X$.
\end{abstract}


\end{frontmatter}



\section{Introduction}
\label{}

\smallskip
Many topological properties are defined or characterized in terms
 of the following classical selection principles given in a general form in \cite{sch}.
 Let $\mathcal{A}$ and $\mathcal{B}$ be sets consisting of
families of subsets of an infinite set $X$. Then:

$S_{1}(\mathcal{A},\mathcal{B})$ is the selection hypothesis: for
each sequence $(A_{n}: n\in \omega)$ of elements of $\mathcal{A}$
there is a sequence $(b_{n}: n\in\omega)$ such that for each $n$,
$b_{n}\in A_{n}$, and $\{b_{n}: n\in\omega \}$ is an element of
$\mathcal{B}$.

$S_{fin}(\mathcal{A},\mathcal{B})$ is the selection hypothesis:
for each sequence $(A_{n}: n\in \omega)$ of elements of
$\mathcal{A}$ there is a sequence $(B_{n}: n\in\omega)$ of finite
sets such that for each $n$, $B_{n}\subseteq A_{n}$, and
$\bigcup_{n\in\omega}B_{n}\in\mathcal{B}$.

$U_{fin}(\mathcal{A},\mathcal{B})$ is the selection hypothesis:
whenever $\mathcal{U}_1$, $\mathcal{U}_2, ... \in \mathcal{A}$ and
none contains a finite subcover, there are finite sets
$\mathcal{F}_n\subseteq \mathcal{U}_n$, $n\in \omega$, such that
$\{\bigcup \mathcal{F}_n : n\in \omega\}\in \mathcal{B}$.

\medskip

 The papers \cite{scheeper,ko,sch,sch1,bts} have initiated the simultaneous
 consideration of these properties in the case where $\mathcal{A}$ and
 $\mathcal{B}$ are important families of open covers of a
 topological space $X$.

\medskip

 An open cover $\mathcal{U}$ of a space $X$ is:

 $\bullet$ an {\it $\omega$-cover} if $X$ does not belong to
 $\mathcal{U}$ and every finite subset of $X$ is contained in a
 member of $\mathcal{U}$.


$\bullet$ a {\it $\gamma$-cover} if it is infinite and each $x\in
X$ belongs to all but finitely many elements of $\mathcal{U}$.


For a topological space $X$ we denote:

$\bullet$ $\mathcal{O}$ --- the family of open covers of $X$;

$\bullet$ $\Omega$ --- the family of open $\omega$-covers of $X$;

$\bullet$ $\Omega_{cz}^{\omega}$ --- the family of countable
cozero $\omega$-covers of $X$;

$\bullet$ $\Gamma$ --- the family of open $\gamma$-covers of $X$;

$\bullet$ $B_{\Omega}$ --- the family of countable Baire (Borel
for a metrizable $X$) $\omega$-covers of $X$;

$\bullet$ $B_\Gamma$ --- the family of countable Baire (Borel for
a metrizable $X$) $\gamma$-covers of $X$.

\medskip In this paper we prove that if every finite power of $X$ is Lindel$\ddot{e}$of then
$C_p(X)$ to be strongly sequentially separable iff  $X$ is
$\gamma$-set ($X$ satisfies $S_1(\Omega,\Gamma)$).

\medskip In Section 4 we prove that the space $B_{\alpha}(X)$ of functions Baire class
$\alpha$  ($1<\alpha\leq\omega_1$) on a Tychonoff space $X$ with
the pointwise topology to be sequentially  separable iff there
 exists a Baire isomorphism class $\alpha$ from a space $X$ onto a
 $\sigma$-set. Also we get criterion of strongly sequentially
 separableness of a space $B_{\alpha}(X)$ ($0<\alpha\leq\omega_1$).

\section{Main definitions and notation}

 We will be denoted by

$\bullet$  $B_{\alpha}(X)$ a set of all functions of Baire class
$\alpha$ for $\alpha\in [0,\omega_1]$ defined
 on a Tychonoff space $X$, provided with the pointwise convergence topology.

In particular:

$\bullet$ $C_{p}(X)=B_0(X)$ a set of all real-valued continuous
functions $C(X)$ defined
 on a Tychonoff space $X$.

$\bullet$ $B_1(X)$ a set of all first Baire class
 functions $B_1(X)$ i.e.,  pointwise limits of continuous functions, defined
 on a Tychonoff space $X$.

$\bullet$ $B(X)=B_{\omega_1}(X)$ a  set of all Baire functions,
defined on a Tychonoff space $X$. If $X$ is metrizable space, then
$B(X)$ be called a space of Borel functions.

\medskip
Recall that the $i$-weight $iw(X)$ of a space $X$ is the smallest
infinite cardinal number $\tau$ such that $X$ can be mapped by a
one-to-one continuous mapping onto a space of the weight not
greater than $\tau$. It well-known that $iw(X)=d(B_{\alpha}(X))$
for $0\leq\alpha\leq \omega_1$ (Theorem 1 in \cite{ps}). In
particular, for every Tychonoff space $X$, $iw(X)=d(C_p(X))$
(\cite{nob}).

  If $X$ is a space and $A\subseteq X$, then the sequential closure of $A$,
 denoted by $[A]_{seq}$, is the set of all limits of sequences
 from $A$. A set $D\subseteq X$ is said to be sequentially dense
 if $X=[D]_{seq}$. If $D$ is a countable sequentially dense subset
 of $X$ then $X$ call sequentially separable space.

 Call $X$ strongly sequentially separable if $X$ is separable and
 every countable dense subset of $X$ is sequentially dense.

 We recall that a subset of $X$ that is the
 complete preimage of zero for a certain function from~$C(X)$ is called a zero-set.
A subset $O\subseteq X$  is called  a cozero-set (or functionally
open) of $X$ if $X\setminus O$ is a zero-set. If a set $Z=\cup_{i}
Z_i$ where $Z_i$ is a zero-set of $X$ for any $i\in \omega$ then
$Z$ is called a $Z_{\sigma}$-set of $X$. Note that if a space $X$
is a perfect normal space, then class of $Z_{\sigma}$-sets of $X$
coincides with class of $F_{\sigma}$-sets of $X$.

It is well known \cite{kur}, that $f\in B_1(X)$ iff $f^{-1}(G)$
--- $Z_{\sigma}$-set for any open set $G$ of real line
$\mathbb{R}$.

Recall that the cardinal $\mathfrak{p}$ is the smallest cardinal
so that there is a collection of $\mathfrak{p}$ many subsets of
the natural numbers with the strong finite intersection property
but no infinite pseudo-intersection. Note that $\omega_1 \leq
\mathfrak{p} \leq \mathfrak{c}$.

For $f,g\in \omega^{\omega}$, let $f\leq^{*} g$ if $f(n)\leq g(n)$
for all but finitely many $n$. $\mathfrak{b}$ is the minimal
cardinality of a $\leq^{*}$-unbounded subset of $\omega^{\omega}$.
A set $B\subset [\omega]^{\infty}$ is unbounded if the set of all
increasing enumerations of elements of $B$ is unbounded in
$\omega^{\omega}$, with respect to $\leq^{*}$. It follows that
$|B|\geq \mathfrak{b}$. A subset $S$ of the real line is called a
$Q$-set if each one of its subsets is a $G_{\delta}$. The cardinal
$\mathfrak{q}$ is the smallest cardinal so that for any $\kappa<
\mathfrak{q}$ there is a $Q$-set of size $\kappa$. (See \cite{do}
for more on small cardinals including $\mathfrak{p}$).

\medskip

Further we use the following theorems.

\begin{theorem} (\cite{vel}). \label{th31} A space $C_p(X)$ is
sequentially separable iff there
 exist  a condensation (one-to-one
continuous map) $f: X \mapsto Y$ from a space $X$ on a
 separable metric space $Y$, such that $f(U)$ --- $F_{\sigma}$-set
 of $Y$ for any cozero-set $U$ of $X$.
\end{theorem}
\medskip

\begin{theorem} (\cite{vel}). \label{th32} A space $B_1(X)$ is
sequentially separable for any separable metric space $X$.
\end{theorem}
\medskip

Note that proof of this theorem  gives more, namely there exists a
countable subset $S\subset C(X)$, such that $[S]_{seq}=B_1(X)$.

   By a {\it set of reals} we usually mean a zero-dimensional, separable
   metrizable space.

\section{Continuous functions}

Recall that $X$ has the property $\gamma$: for any open
$\omega$-cover $\alpha$ of $X$ there is a sequence $\beta \subset
\alpha$ such that $\beta$ is a $\gamma$-cover of $X$.

Recall that $X$ has projectively property $(\gamma)$, if every
continuous second countable image of $X$ has the property $\gamma$
\cite{bcm}.

Note that $S_1(\Omega, \Gamma)$ is equivalent to the
$\gamma$-property introduced by Gerlits and Nagy in \cite{gn}. So
we will call projectively $S_1(\Omega, \Gamma)$ instead of
projectively $(\gamma)$.

In (\cite{bcm}, Theorem 54), M. Bonanzinga, F. Cammaroto, M.
 Matveev proved

 \begin{theorem}\label{th21}  The following conditions are equivalent for a
space $X$:

\begin{enumerate}

\item $X$ $\models$ projective $S_{1}(\Omega,\Gamma)$;

\item every Lindel\"{o}f image of $X$ has property $(\gamma)$;

\item for every continuous mapping $f:X \mapsto
\mathbb{R}^{\omega}$, $f(X)$ $\models$ $S_{1}(\Omega, \Gamma)$;

\item for every continuous mapping $f:X \mapsto \mathbb{R}$,
$f(X)$ $\models$ $S_{1}(\Omega, \Gamma)$;

\item for every countable $\omega$-cover $\mathcal{U}$ of $X$ by
cozero sets, one can pick $U_n\in \mathcal{U}$ so that every $x\in
X$ is contained in all but finitely many $U_n$;

\item $X$ $\models$ $S_{1}(\Omega_{cz}^{\omega},\Gamma)$.

\end{enumerate}

\end{theorem}

\medskip

In \cite{osszts}, authors was proved

\begin{theorem}\label{th3} For a Tychonoff space $X$ the following statements are
equivalent:

\begin{enumerate}

\item $C_p(X)$ is strongly sequentially separable;

\item $X$ $\models$ projective $S_{1}(\Omega,\Gamma)$ and
$iw(X)=\aleph_0$.

\end{enumerate}

\end{theorem}

By Proposition 55 in \cite{glm}, we have that

$\bullet$ every projectively $(\gamma)$ space is zero-dimensional;

$\bullet$ every space of cardinality less than $\mathfrak{p}$ is
projectively $(\gamma)$;

$\bullet$ the projectively $(\gamma)$ property is preserved by
continuous images.

Then we have the next

\begin{proposition} Let $X$ be a Tychonoff space with
$iw(X)=\aleph_0$.

$\bullet$ If $C_p(X)$ is strongly sequentially separable, then $X$
is zero-dimensional.

$\bullet$ If cardinality of $X$ less than $\mathfrak{p}$, then
$C_p(X)$ is strongly sequentially separable.

$\bullet$ If $C_p(X)$ is strongly sequentially separable and
$h:X\mapsto Y$ is continuous mapping from $X$ onto a space $Y$
with $iw(Y)=\aleph_0$, then $C_p(Y)$ is also strongly sequentially
separable.

\end{proposition}

\begin{proposition} Suppose a space $X$ $\models$ projective $S_{1}(\Omega,\Gamma)$,
and $F\subseteq X$ be a $Z_{\sigma}$-set of $X$. Then $F$
$\models$ projective $S_{1}(\Omega,\Gamma)$.

\end{proposition}

\begin{proof}

Let $\mathcal{V}_n=\{V^n_s : s\in \omega\}\in
\Omega_{cz}^{\omega}$ of $F$ for each $n\in \omega$, and
$F=\bigcup\limits_i F_i$, where $F_i$ is a zero-set of $X$ and
$F_{i}\subset F_{i+1}$ for each $i\in \omega$.

Consider $\mathcal{V}_1=\{V^1_s : s\in \omega\}$ and $i_0=1$.
 For each $s\in \omega$ there is
$F_{i_s}$ such that $F_{i_s}\nsubseteq V^1_s$ and $i_{s}>i_{s-1}$.
Let $\mathcal{W}_1=\{V^1_s\cup (X\setminus F_{i_s}) : s\in
\omega\}$.

Fix $n\in \omega\setminus\{1\}$ consider $\mathcal{V}_n=\{V^n_s
:s\in \omega\}$ and $i_0=n$.
 For each $s\in \omega$ there is
$F_{i_s}$ such that $F_{i_s}\nsubseteq V^n_s$ and $i_{s}>i_{s-1}$.
Let $\mathcal{W}_n=\{W^n_s:=V^n_s\cup (X\setminus F_{i_s}) :s\in
\omega\}$. Claim that $\mathcal{W}_n\in \Omega_{cz}^{\omega}$ of
$X$
 for each $n\in \omega$. Let $K$ be a finite subset of $X$. Since $X\setminus F\subset
W^n_s$ for $s\in \omega$ we can assume that $K\subset F$. There
exists $s\in \omega$ such that $K\subset V^n_s$. Then $K\subset
W^n_s$.

Since $X$ has the property $S_{1}(\Omega_{cz}^{\omega}, \Gamma)$
 there is the sequence $\mathcal{S}=\{W^n_{s(n)}=V^n_{s(n)}\cup (X\setminus F_{i_{s(n)}}) : n\in
\omega\}$ such that $W^n_{s(n)}\in \mathcal{W}_n$ and
$\mathcal{S}$ is a $\gamma$-cover of $X$.

We claim that  $\{V^n_{s(n)} : n\in \omega\}$ is a $\gamma$-cover
of $F$.

Let $P$ be a finite subset of $F$. There is $i'\in \omega$ such
that $P\subset F_{i'}$. Since $\mathcal{S}$ is a $\gamma$-cover of
$X$ there is $n'$ such that $n'>i'$ and $P\subset W^{n}_{s(n)}$
for each $n>n'$, hence, $P\subset V^{n}_{s(n)}$ for each $n>n'$.

\end{proof}

\begin{corollary} Let $C_p(X)$ be strongly sequentially
separable for a Tychonoff space $X$, and $F\subseteq X$ be a
$Z_{\sigma}$-set of $X$. Then $C_p(F)$ is strongly sequentially
separable.

\end{corollary}

\begin{theorem}\label{tpr1} Let $X$ be a Tychonoff space,  $iw(X)=\aleph_0$ and $A\subseteq X$.
The space $(X\setminus A)\bigsqcup A$ $\models$ projective
$S_{1}(\Omega,\Gamma)$ iff a space $X$ $\models$ projective
$S_{1}(\Omega,\Gamma)$, and $A$ and $X\setminus A$ are
$Z_{\sigma}$-sets in $X$.
\end{theorem}

\begin{proof} $(1)\Rightarrow(2)$. This implication may be proved in much the same way as Theorem 5  in
\cite{gami}. Since $iw(X)=\aleph_0$ there are countable cozero
family $\gamma$ in $X$ such that for each $F\in [(X\setminus
A)\bigsqcup A]^{<\omega}$ there exists functionally separated (in
$X$) subsets $C_F$, $D_F\in \gamma$ such that $F\subset
C_F\bigsqcup D_F\subset(X\setminus A)\bigsqcup A$.

By the countable $\gamma$-property there exists $F_n$ for $n\in
\omega$ such that

 $(X\setminus A)\bigsqcup A\subset
\bigcup\limits_{n} \bigcap\limits_{m>n} (C_{F_m}\bigsqcup
D_{F_m})$.

Since $C_{F_n}$ and $D_{F_n}$ are functionally separated sets in
$X$ (i.e., there is $f_n\in C(X)$ such that $f_n^{-1}(0)\supseteq
C_{F_n}$ and $f_n^{-1}(1)\supseteq D_{F_n}$)

 $\bigcup\limits_{n}
\bigcap\limits_{m>n} f_m^{-1}(0)$ and $\bigcup\limits_{n}
\bigcap\limits_{m>n} f_m^{-1}(1)$ are disjoint, and they show that
$X\setminus A$ and $A$ are $Z_{\sigma}$-sets in $X$.

The mapping $id: (X\setminus A)\bigsqcup A \mapsto X$ is
continuous, hence, $X$ is a $\gamma$-set.

$(2)\Rightarrow(1)$. Consider a countable cozero $\omega$-cover
$\alpha=\{V_k\}$ of $(X\setminus A)\bigsqcup A$. Let
$A=\bigcup\limits_i F_i=\bigcap\limits_j G_j$ where $F_i$ is a
zero-set of $X$ and $G_j$ is an cozero-set of $X$ for each $i,j\in
\omega$. Denote $\widetilde{V_k}=V_k\cup (G_k\setminus F_k)$ for
each $k\in \omega$. Note that $\{\widetilde{V_k}\}$ is a countable
cozero $\omega$-cover of $X$. There is a sequence
$\beta=\{\widetilde{V_{k_s}}\} \subset \{\widetilde{V_k}\}$ such
that $\beta$ is a $\gamma$-cover of $X$. It follows that
$\{V_{k_s}\}$ is a $\gamma$-cover of $(X\setminus A)\bigsqcup A$.

\end{proof}

\medskip
For $X\subseteq [0,1]$ let $X+1=\{x+1 : x\in X\}$.

\medskip

\begin{corollary}(Theorem 5 in \cite{gami}) \label{pr1} Suppose $A\subseteq X\subseteq [0,1]$. A space
$(X\setminus A)\cup (A+1)$ is a $\gamma$-set iff the space $X$ is
a $\gamma$-set and $A$ is $G_{\delta}$ and $F_{\sigma}$ in $X$.
\end{corollary}

Note that property $(\gamma)$ is preserved by finite powers
\cite{g1}, but for projective $S_{1}(\Omega,\Gamma)$ is not the
case (Example 58 in \cite{bcm}).

\medskip
\begin{proposition}\label{pr74} Suppose $X$ $\models$ projective
$S_{1}(\Omega,\Gamma)$. Then $X\bigsqcup X$ $\models$ projective
$S_{1}(\Omega,\Gamma)$.

\end{proposition}

\begin{proof} Let $\mathcal{U}=\{U_i: i\in \omega\}$ be a countable $\omega$-cover  of $X\bigsqcup X$ by
cozero sets. Let $X\bigsqcup X=X_1\bigsqcup X_2$ where $X_i=X$ for
$i=1,2$. Consider $\mathcal{V}_1=\{U^1_i=U_i\bigcap X_1 :
X_1\setminus U_i\neq\emptyset, i\in \omega\}$ and
$\mathcal{V}_2=\{U^2_i=U_i\bigcap X_2 : X_2\setminus
U_i\neq\emptyset, i\in \omega \}$ as families of subsets of the
space $X$. Define $\mathcal{V}:=\{ U^1_i\bigcap U^2_i : U^1_i\in
\mathcal{V}_1$ and $U^2_i\in \mathcal{V}_2 \}$. Note that
$\mathcal{V}$ is a countable $\omega$-cover  of $X$ by cozero
sets. By Theorem \ref{th21}, there is $\{U^1_{i_n}\bigcap
U^2_{i_n} :n\in \omega\}\subset \mathcal{V}$ such that
$\{U^1_{i_n}\bigcap U^2_{i_n} :n\in \omega\}$ is a $\gamma$-cover
of $X$. It follows that $\{U_{i_n}:n\in \omega \}$ is a
$\gamma$-cover of $X\bigsqcup X$.

\end{proof}

\begin{proposition}
Suppose $X$ $\models$ projective $S_{1}(\Omega,\Gamma)$, $A$ and
$X\setminus A$ are $Z_{\sigma}$-sets in $X$. Then $X\bigsqcup A$
$\models$ projective $S_{1}(\Omega,\Gamma)$.
\end{proposition}

\begin{proof}  By Theorem \ref{tpr1}, $(X\setminus A)\bigsqcup A$ $\models$ projective $S_{1}(\Omega,\Gamma)$. Let $Y=((X_1\setminus
A_1)\bigsqcup A_1)\bigsqcup ((X_2\setminus A_2)\bigsqcup A_2))$
where $X_i=X$, $A_i=A$ for $i=1,2$. By Proposition \ref{pr74}, $Y$
is projective $S_{1}(\Omega,\Gamma)$.

 Define the continuous mapping
$f: Y \mapsto X\bigsqcup A$ defined by $f$

$$f= \left\{
\begin{array}{rcl}
id(X_1\setminus A_1)=X\setminus X_A \,\, \, \, \, \,  \,  \\
id(A_1)=X_A \, \, \, \, \, \, \, \, \, \, \, \, \, \, \,  \,  \\
id(X_2\setminus A_2)=X\setminus X_A \,\, \, \, \, \,  \,  \\
id(A_2)=A \,\, \, \, \, \,  \, \, \, \, \, \,  \,

\end{array}
\right.
$$

where $X_A\subset X$ such that $X_A=A$ and $id$ is an identity
mapping. Note that the projectively $(\gamma)$ property is
preserved by continuous images (Proposition 55 in \cite{glm}).
Since $Y$ $\models$ projective $S_{1}(\Omega,\Gamma)$ and $f$ is a
continuous mapping this implies $X\bigsqcup A$ $\models$
projective $S_{1}(\Omega,\Gamma)$.

\end{proof}

\begin{corollary} \label{cr1}Suppose $X$ is a $\gamma$-set, $A$ is $G_{\delta}$ and
$F_{\sigma}$ in $X$. Then $X\bigsqcup A$ is a $\gamma$-set.
\end{corollary}

\begin{proposition}
Suppose a Tychonoff space $X$ such that $C_p(X)$ is strongly
sequentially separable, $A\subset X$ and $A$ and $X\setminus A$
are $Z_{\sigma}$-sets in $X$. Then $C_p(X\bigsqcup A)$ is strongly
sequentially separable.
\end{proposition}

\begin{corollary} \label{cr1} Suppose $X$ is a perfect normal space and $C_p(X)$ is strongly
sequentially separable, $A$ is $G_{\delta}$ and $F_{\sigma}$ in
$X$. Then $C_p(X\bigsqcup A)$  is strongly sequentially separable.
\end{corollary}

Recall that $l^{*}(X)\leq \aleph_0$ if every finite power of $X$
is Lindel$\ddot{e}$of (or, by Proposition in \cite{gn}, if every
$\omega$-cover of $X$ contains an at most countable
$\omega$-subcover of $X$).

\begin{theorem}\label{th100} For a Tychonoff space $X$ with $l^{*}(X)\leq \aleph_0$ and $n\in \omega$ the following statements are
equivalent:

\begin{enumerate}

\item $C_p(X)$ is strongly sequentially separable;

\item $C_p(X^n)$ is strongly sequentially separable;

\item $(C_p(X))^{\aleph_0}$ is strongly sequentially separable;

\item $C_p(X)$ is separable and Frechet-Urysohn;

\item $C_p(X^n)$ is separable and Frechet-Urysohn;

\item $(C_p(X))^{\aleph_0}$ is separable and Frechet-Urysohn;

\item $iw(X)=\aleph_0$ and $X$ $\models$ projective
$S_{1}(\Omega,\Gamma)$;

\item $iw(X)=\aleph_0$ and $X^n$ $\models$ projective
$S_{1}(\Omega,\Gamma)$;

\item $iw(X)=\aleph_0$ and $X$ $\models$ $S_{1}(\Omega, \Gamma)$;

\item $iw(X^n)=\aleph_0$ and $X^n$ $\models$ $S_{1}(\Omega,
\Gamma)$;

\item $iw(X)=\aleph_0$ and $X$ has $\gamma$ property;

\item $iw(X^n)=\aleph_0$ and $X^n$ has $\gamma$ property;

\item $C_p(X, \mathbb{R}^{\aleph_0})$ is separable and
Frechet-Urysohn;

\item $C_p(X, \mathbb{R}^{\aleph_0})$ is strongly sequentially
separable.

\end{enumerate}

\end{theorem}

\medskip

\begin{proof} By Theorem \ref{th3}, Proposition \ref{pr74} and Theorem II.3.2 in
\cite{arh}, and that $(C_p(X))^{\aleph_0}$ is homeomorphic to the
space $C_p(X, \mathbb{R}^{\aleph_0})$ (\cite{arch10}).

\end{proof}

\medskip

By Todor$\check{c}$evi$\acute{c}$ Theorem (Theorem 4 in
\cite{gami}) and Theorem \ref{th100} we have the next

\begin{proposition} Assuming $\diamondsuit_{\omega_1}$ there
exists a $\gamma$-set $X$ of cardinality  $\omega_1=\mathfrak{c}$
that for every subset $Y$ of $X$ the space $C_p(Y)$ is strongly
sequentially separable.
\end{proposition}

\section{Baire functions class $\alpha$}

Let $X$ be a Tychonoff space and $C(X)$ the space of continuous
real-valued functions on $X$. Let $B_0(X)=C(X)$, and inductively
define $B_{\alpha}(X)$ for each ordinal $\alpha\leq \omega_1$ to
be the space of pointwise limits of sequences of functions in
$\bigcup\limits_{\xi<\alpha} B_{\xi}(X)$. Let $B^*_{\alpha}(X)$ be
the space of bounded functions in $B_{\alpha}(X)$.

The functions in
$B(X)=B_{\omega_1}(X)=\bigcup\limits_{\alpha<\omega_1}
B_{\alpha}(X)$ are called Baire functions or, if $X$ is
metrizable, Borel functions.

The Baire sets of $X$ of multiplicative class $\alpha$, denoted by
$Z_{\alpha}(X)$, are defined to be the zero sets of functions in
$B^*_{\alpha}(X)$. Those of additive class $\alpha$, denoted by
$CZ_{\alpha}(X)$, are defined as the complements of sets in
$Z_{\alpha}(X)$. $Z_{\omega_1}(X)=\bigcup\limits_{\alpha<\omega_1}
Z_{\alpha}(X)$. Finally, those of ambiguous class $\alpha$,
denoted by $A_{\alpha}(X)$, are the sets which are simultaneously
in $Z_{\alpha}(X)$ and $CZ_{\alpha}(X)$.

With the set-theoretic operations of unions and intersection,
$A_{\alpha}(X)$ is a Boolean algebra for each $\alpha\leq
\omega_1$. By the Lebesgue-Hausdoff Theorem (Theorem 6.1.1 in
\cite{jr}), for each $\alpha<\omega_1$

$Z_{\alpha+1}(X)=(CZ_{\alpha}(X))_{\delta}$, and
$CZ_{\alpha+1}(X)=(Z_{\alpha}(X))_{\sigma}$; and if $\lambda$ is a
limit ordinal, then $Z_{\lambda}(X)=(\bigcup\limits_{\xi<\lambda}
CZ_{\xi}(X))_{\sigma \delta}$ and
$CZ_{\lambda}(X)=(\bigcup\limits_{\xi<\lambda} Z_{\xi}(X))_{\delta
\sigma}$.

It is well known (\cite{kur}), that $f\in B_{\alpha}(X)$ iff
$f^{-1}(G)\in CZ_{\alpha}(X)$ for any open set $G$ of real line
$\mathbb{R}$.

In \cite{ospy}, Osipov and Pytkeev have established criterion for
$B_{1}(X)$ to be sequentially separable.

\begin{theorem}(Osipov, Pytkeev)\label{th6}
 A function
space $B_1(X)$ is sequentially separable iff there
 exists a bijection $\varphi: X \mapsto Y$ from a space $X$ onto a
 separable metrizable space $Y$, such that

\begin{enumerate}

\item $\varphi^{-1}(U)$ --- $Z_{\sigma}$-set of $X$ for any open
set $U$ of $Y$;

\item  $\varphi(T)$ --- $F_{\sigma}$-set of $Y$ for any zero-set
$T$ of $X$.

\end{enumerate}

\end{theorem}

\medskip
\begin{corollary}\label{th91} A function space $B_1(X)$ is sequentially separable iff there
 exists a Baire isomorphism class 1 from a space $X$ onto a
separable metrizable space.

\end{corollary}

\medskip

\begin{theorem}\label{th9}
 For any $1<\alpha\leq \omega_1$, a function space $B_{\alpha}(X)$ is sequentially separable iff there
 exists a bijection $\varphi: X \mapsto Y$ from a space $X$ onto a
 separable metrizable space $Y$, such that

\begin{enumerate}

\item $\varphi^{-1}(U)\in CZ_{\alpha}(X)$ for any open set $U$ of
$Y$;

\item  $\varphi(T)$ --- $F_{\sigma}$-set of $Y$ for  $T\in
(\bigcup\limits_{\xi<\alpha} Z_{\xi}(X))_{\delta}$.

\end{enumerate}

\end{theorem}

\begin{proof} $(1) \Rightarrow (2)$. Let $B_{\alpha}(X)$ be sequentially
separable, and $S$ be a countable sequentially dense subset
 of $B_{\alpha}(X)$. Consider the topology $\tau$ generated by the family
 $\mathcal{P}=\{f^{-1}(G): G$ is an open set of $\mathbb{R}$ and
 $f\in S \}$. A space $Y=(X,\tau)$ is a separable metrizable space
 because $S$ is a countable dense subset
 of $B_{\alpha}(X)$. Note that a function $f\in S$, considered as mapping from
 $Y$ to $\mathbb{R}$, is a continuous function. Let $\varphi$ be the identity
 mapping from $X$ on $Y$.

 We claim that $\varphi^{-1}(U)\in CZ_{\alpha}(X)$ for any open
set $U$ of $Y$.

Note that $CZ_{\alpha}(X)$  is closed under a countable unions and
a finite intersections. It follows that it is sufficient to prove
for any $P\in \mathcal{P}$. But $\varphi^{-1}(P)\in
CZ_{\alpha}(X)$  because $f\in S\subset B_{\alpha}(X)$.

Let $T\in (\bigcup\limits_{\xi<\alpha} Z_{\xi}(X))_{\delta}\subset
A_{\alpha}(X)$ and $h$ be a characteristic function of $T$. Since
$T\in A_{\alpha}(X)$, $h\in B_{\alpha}(X)$ (see Theorem 1, \S 31
in \cite{kur}).

There exists $\{f_n\}_{n\in \omega}\subset S$ such that
$\{f_n\}_{n\in \omega}$ converges to $h$. Since $S\subset C_p(Y)$,
$h\in B_1(Y)$ and, hence, $h^{-1}(\frac{1}{2},\frac{3}{2})=T$ is a
$Z_{\sigma}$-set of $Y$. (Note that if a space $Z$ is a perfect
normal space, then class of $Z_{\sigma}$-sets of $Z$ coincides
with class of $F_{\sigma}$-sets of $Z$).

$(2) \Rightarrow (1)$. Let $\varphi$ be a bijection from $X$ on
$Y$ satisfying the conditions of theorem. Then $h=f\circ
\varphi\in B_{\alpha}(X)$ for any $f\in C(Y)$
($h^{-1}(G)=\varphi^{-1}(f^{-1}(G))\in CZ_{\alpha}(X)$ for any
open set $G$ of $\mathbb{R}$).

Moreover, $g=f\circ \varphi^{-1}\in B_1(Y)$ for any $f\in
B_{\alpha}(X)$ because of $\varphi(Z)$ is a $F_{\sigma}$-set of
$Y$ for any  $Z\in CZ_{\alpha}(X)$. Define a mapping $F:
B_{\alpha}(X) \mapsto B_1(Y)$ by $F(f)=f\circ \varphi^{-1}$.

Since $\varphi$ is a bijection, $C(Y)$ embeds in
$F(B_{\alpha}(X))$ i.e., $C(Y)\subseteq F(B_{\alpha}(X))\subseteq
B_1(Y)\subseteq B_{\alpha}(Y)\subseteq F(B_{\alpha}(X))$. Note
that $F(B_{\alpha}(X))=B_1(Y)=B_{\alpha}(Y)$ i.e., $Y$ is a
$\sigma$-set (recall that a set of reals $X$ is $\sigma$-set if
each $G_{\delta}$ subset of $X$ is an $F_{\sigma}$ subset of $X$
\cite{kur}).

By Theorem 1 in \cite{vel}, each subspace $D$ such that
$C(Y)\subset D\subset B_1(Y)$ is sequentially separable. Thus
$F(B_{\alpha}(X))$ is sequentially separable, and, hence,
$B_{\alpha}(X)$ is sequentially separable.
\end{proof}

\medskip


\begin{corollary}\label{th94}
 A function space $B(X)$ is sequentially separable iff there
 exists a bijection $\varphi: X \mapsto Y$ from a space $X$ onto a
 separable metrizable space $Y$, such that

\begin{enumerate}

\item $\varphi^{-1}(U)$ --- Baire set of $X$ for any open set $U$
of $Y$;

\item  $\varphi(T)$ --- $F_{\sigma}$-set of $Y$ for any Baire set
$T$ of $X$.

\end{enumerate}

\end{corollary}

A bijection $f: X \mapsto Y$ between Tychonoff spaces is called a
Baire isomorphism if $f(Z_{\omega_1}(X))=Z_{\omega_1}(Y)$; and is
said to be of class $(\alpha$,$\beta)$ (or $\alpha$ if
$\alpha=\beta$) if $f^{-1}(Z_0(Y))\subset Z_{\alpha}(X)$ and
$f(Z_0(X))\subset Z_{\beta}(Y)$.

If $X$ and $Y$ are metrizable, then $f$ is usually called a Borel
isomorphism.

\medskip

Note that if $X$ is a $\sigma$-set, then $B_1(X)=B(X)$ is
sequentially separable (Theorem \ref{th32}).

\medskip

Recall that $X\subset 2^{\omega}$ is Sierpi$\acute{n}$ski set, if
it is uncountable, but for every measure zero set $M$, $X\bigcap
M$ is countable. Sierpi$\acute{n}$ski showed that the Continuum
Hypothesis ($CH$) implies the existence of such sets.

\begin{proposition} (CH) Let $X$ be a
Sierpi$\acute{n}$ski set. Then $B_1(X)=B(X)$ is sequentially
separable.
\end{proposition}

\begin{proof} By Szpilrajn (Marczewski) Theorem (see in \cite{sm}), if $X$ is a
Sierpi$\acute{n}$ski set, then $X$ is a $\sigma$-set.
\end{proof}

\begin{corollary}\label{th92} Let  $X$ be a Tychonoff space and $1<\alpha\leq \omega_1$. A function space $B_{\alpha}(X)$ is sequentially separable
 iff there
 exists a Baire isomorphism of class $\alpha$ from a space $X$ onto
 a $\sigma$-set.
\end{corollary}

\begin{corollary}\label{th93} Let  $X$ be a metrizable space. A space $B(X)$ of Borel functions is sequentially separable iff there
 exists a Borel isomorphism  from a space $X$ onto
 a $\sigma$-set.
\end{corollary}

Recall that Baire order $ord(C(X))=\min \{\beta :
B_{\beta}(X)=B_{\beta+1}(X)\}$.

\begin{corollary}\label{cr1} Suppose that $B_{\alpha}(X)$ is sequentially separable for a Tychonoff space $X$ and $1<\alpha\leq \omega_1$.
Then Baire order $ord(C(X))\leq \alpha$ and $ord(C(X))<\omega_1$.
\end{corollary}

\begin{proof} If exists a Baire isomorphism $F$ of class $\alpha$ from a space $X$ onto
 a $\sigma$-set $Y$, then $F$ is a Baire isomorphism $F$ of class
 $\beta$ for every $\beta \geq \alpha$. By Theorem
 \ref{th9}, $F(B_{\alpha}(X))=B_1(Y)=F(B_{\beta}(X))$ for every $\beta \geq
 \alpha$. It follows that $ord(C(X))\leq \alpha$.

Note that if $B(X)$ is sequentially separable, then there exists a
countable sequentially dense subset $S\subset B(X)$, and, hence,
$S\subset B_{\alpha'}(X)$ for some $\alpha'<\omega_1$. It follows
that $B_{\alpha'}(X)$ is also sequentially separable,
$B_{\alpha'}(X)=B(X)$ and $ord(C(X))<\omega_1$.

\end{proof}

 The basic existence result are given in the next theorem.

 \begin{theorem}(Lebesgue, \cite{leb})\label{th666} If $X$ is a complete metric space
 containing a non-empty perfect subset, then for all $\alpha<
 \omega_1$, $B_{\alpha+1}(X)\setminus B_{\alpha}(X)\neq \emptyset$.
\end{theorem}

\begin{theorem}($\check{C}$oban, \cite{cob}, Jayne, \cite{j})\label{th667} If $X$ is a compact Hausdorff space
 containing a non-empty perfect subset, then for all $\alpha<
 \omega_1$,

 $B_{\alpha+1}(X)\setminus B_{\alpha}(X)\neq \emptyset$.
\end{theorem}

\begin{theorem}(P.R. Meyer, \cite{mey})\label{th668} If $X$ is a compact Hausdorff space which contains no non-empty perfect subsets, then $B_{1}(X)=B_{2}(X)$.
\end{theorem}

Note that, if $X$ is an uncountable Polish space, then $B_1(X)$ is
sequentially separable (Theorem \ref{th32}), but, by Lebesgue
Theorem \ref{th666} and Corollary \ref{cr1}, $B_{\alpha}(X)$ is
not sequentially separable for $1<\alpha\leq \omega_1$. The same
is true for any uncountable analytic ($\bf \sum^1_1$) space $X$
since $X$ has a perfect subspace (see \cite{kur}).

By Theorem \ref{th668}, if $X$ is a metrizable compact space which
contains no non-empty perfect subsets, then $B_{1}(X)=B(X)$ is
sequentially separable.

\begin{proposition} There exists an example (an uncountable Polish space containing a non-empty perfect subset) a space $X$, such that
$B_1(X)$ is sequentially separable, but $B_{\alpha}(X)$ is not
sequentially separable for all $1<\alpha\leq \omega_1$.
\end{proposition}

\medskip

A natural generalization of a $\sigma$-set is the notion of
$Q$-set. A set of reals $X$ is a $Q$-set if every subset of $X$ is
a relative $F_{\sigma}$. Assuming Martin's axiom ($MA$) every set
of reals of cardinality less than the continuum is a $Q$-set
(Theorems in \cite{ms} or \cite{ru}).

\medskip

\begin{proposition}(MA) Let $X$ be a set of reals,
$|X|<\mathfrak{c}$. Then $B_{1}(X)=R^{X}$ is sequentially
separable.

\end{proposition}

\begin{proposition}(MA) Let $X$ be a Tychonoff space, $iw(X)=\aleph_0$ and
$|X|<\mathfrak{c}$. Then $B(X)$ is sequentially separable.
\end{proposition}

By Theorem 9.1 in \cite{do}, $\mathfrak{b}=\min\{|X|: X$ is a
separable metrizable, but it is not a $\sigma$-set$\}$. Then we
have a next result.

\begin{proposition} Let $X$ be a Tychonoff space, $iw(X)=\aleph_0$ and
$|X|<\mathfrak{b}$. Then $B(X)$ is sequentially separable.
\end{proposition}

On the other hand, it is consistent that there are no uncountable
$\sigma$-sets, in fact, it is consistent that every uncountable
set of reals has Baire order $\omega_1$ (Theorem 22 in
\cite{mil2}).

\medskip

\begin{proposition} It is  consistent that there are no uncountable
$\sigma$-sets, if $B_{\alpha}(X)$ is sequentially separable for a
Tychonoff space $X$ and $\alpha>1$, then $X$ is countable.

\end{proposition}

We denote:

$\bullet$ $Z^{\alpha}=(\bigcup\limits_{\xi<\alpha}
Z_{\xi}(X))_{\delta}$;

 $\bullet$ $Z^{\alpha}_{\Omega}$ --- the family of countable
 $Z^{\alpha}$
 $\omega$-covers of $X$;

$\bullet$ $Z^{\alpha}_\Gamma$ --- the family of countable
$Z^{\alpha}$ $\gamma$-covers of $X$.

A set $X$ is called a $Z^{\alpha}$-cover $\gamma$-set iff every
countable $\omega$-cover of $X$ by $Z^{\alpha}$ sets contains a
$\gamma$-cover.

 Being a $Z^{\alpha}$-cover $\gamma$-set is equivalent to saying that for
 any $\omega$-sequence of countable $Z^{\alpha}$ $\omega$-covers of $X$
 we can choose one element from each and get a $\gamma$-cover of
 $X$ -- this is denoted $S_1(Z^{\alpha}_{\Omega},Z^{\alpha}_{\Gamma})$.

\medskip

A standard diagonalization trick gives the following.

\begin{proposition}\label{pr104} A $Z^{\alpha}$-cover $\gamma$-set $X$ is equivalent
to $X$ $\models$ $S_{1}(Z^{\alpha}_{\Omega},
Z^{\alpha}_{\Gamma})$.
\end{proposition}

\begin{proof} The proof of this is like that of the corresponding result in
\cite{gn}.

\end{proof}

We recall  concept $Z_{\sigma}$-mapping:

a map $f:X \mapsto Y$ be called a $Z_{\sigma}$-map, if $f^{-1}(Z)$
is a $Z_{\sigma}$-set of $X$ for any zero-set $Z$ of $Y$.

\medskip





\medskip


\medskip
In \cite{ospy}, Osipov and Pytkeev have established criterion for
$B_{1}(X)$ to be strongly sequentially separable.

\begin{theorem}(Osipov, Pytkeev) \label{th5} A function
space $B_1(X)$ is strongly sequentially separable iff $X$ has a
coarser second countable topology, and for any bijection $\varphi$
from a space $X$ onto a
 separable metrizable space $Y$, such that  $\varphi^{-1}(U)$ --- $Z_{\sigma}$-set of $X$ for any open set
$U$ of $Y$, the space $Y$ has the property $\gamma$.
\end{theorem}

\begin{theorem}\label{th25} For a Tychonoff space $X$ and $0<\alpha\leq \omega_1$, the following statements are
equivalent:

\begin{enumerate}

\item $B_{\alpha}(X)$ is strongly sequentially separable;

\item $iw(X)=\aleph_0$ and $X$ $\models$
$S_{1}(Z^{\alpha}_{\Omega}, Z^{\alpha}_{\Gamma})$;

\item $iw(X)=\aleph_0$ and $X$ is a $Z^{\alpha}$-cover
$\gamma$-set.

\end{enumerate}

\end{theorem}

\begin{proof} $(1)\Rightarrow(3)$. Let $B_{\alpha}(X)$ be strongly sequentially
separable space and $\alpha=\{V_i\}$ be a countable $Z^{\alpha}$
$\omega$-cover of $X$. Let $S=\{h_i: i\in \omega \}$ be a
countable sequentially dense subset of $B_{\alpha}(X)$. Consider
the countable set $D=\{ f_{i,j}\in B_{\alpha}(X) :
f_{i,j}\upharpoonright V_i=h_j$ and $f_{i,j}\upharpoonright
(X\setminus V_i)=1$ for $i,j\in \omega \}$. Since $\alpha=\{V_i\}$
is $Z^{\alpha}$ $\omega$-cover of $X$ and $S$ is a dense subset of
$B_{\alpha}(X)$, the set $D$ is a dense subset of $B_{\alpha}(X)$.
By $(1)$, the set $D$ is a countable sequentially dense subset of
$B_{\alpha}(X)$. Then there exists a sequence
$\{f_{i_k,j_k}\}_{k\in \omega}$ converge to $\bf{0}$ such that
$f_{i_k,j_k}\in D$ for every $k\in \omega$. Claim that the
sequence $\{V_{i_k} : k\in \omega\}$ is a $\gamma$-cover of $X$.
Let $K$ be a finite subset of $X$ and $W=[K,(-1,1)]$ is a base
neighborhood of $\bf{0}$, then there is $k'\in \omega$ such that
$f_{i_k,j_k}\in W$ for each $k>k'$. It follows that $K\subset
V_{i_k}$ for each $k>k'$. We thus get that $X$ is a
$Z^{\alpha}$-cover $\gamma$-set.

$(2)\Leftrightarrow(3)$. By Proposition \ref{pr104}.

$(2)\Rightarrow(1)$ Let $iw(X)=\aleph_0$, $X$ $\models$
$S_{1}(Z^{\alpha}_{\Omega}, Z^{\alpha}_{\Gamma})$ and $D=\{d_i:
i\in \omega\}$ be a countable dense subset of $B_{\alpha}(X)$.
Claim that there exists a sequence $\{d_{i_k}\}_{k\in \omega}$
such that $d_{i_k}\in D$ for every $k\in \omega$ and
$\{d_{i_k}\}_{k\in \omega}$ converge to $\bf{0}$.

The set $\mathcal{V}_j=\{V^i_j=d_i^{-1}(-\frac{1}{j},\frac{1}{j}):
d_i\in D\}$ is a countable $\omega$-cover of $X$ by $Z^{\alpha}$
sets for every $j\in \omega$. Indeed,  let $K$ be a finite subset
of $X$ and $W=[K, (-\frac{1}{j},\frac{1}{j})]$ be a base
neighborhood of $\bf{0}$, then there is $d\in D$ such that $d\in
W$, hence, $K\subset d^{-1}(-\frac{1}{j},\frac{1}{j})$.

By (2), there is a sequence $\{V_j^{i(j)}\}_{j\in \omega}$ such
that $V_j^{i(j)}\in \mathcal{V}_j$ and $\{V^{i(j)}_j: j\in
\omega\}$ is a $\gamma$-cover of $X$. Claim that
$\{d_{i(j)}\}_{j\in \omega}$ converge to $\bf{0}$. Let $K$ be a
finite subset of $X$, $\epsilon>0$ and $O=[K,
(-\epsilon,\epsilon)]$ be a base neighborhood of $\bf{0}$, then
there is $j'\in \omega$ such that $\frac{1}{j'}<\epsilon$ and
$K\subset V_j^{i(j)}$ for $j>j'$. It follows that $d_{i(j)}\in O$
for $j>j'$.

\end{proof}

\medskip

Recall that a set $X$ is called a Borel-cover $\gamma$-set iff
every countable $\omega$-cover of $X$ by Borel sets contains a
$\gamma$-cover (see in \cite{mil1}).

 Being a Borel-cover $\gamma$-set is equivalent to saying that for
 any $\omega$-sequence of countable Borel $\omega$-covers of $X$
 we can choose one element from each and get a $\gamma$-cover of
 $X$ -- this is denoted $S_1(B_{\Omega},B_{\Gamma})$. The
 equivalence was proved by Gerlitz and Nagy \cite{gn} for open
 covers i.e., for $S_1(\Omega,\Gamma)$, but the proof works also
 for Borel covers as was noted in Scheepers and Tsaban \cite{scts}.

\begin{proposition}(Scheepers, Tsaban)\label{pr1} A Borel-cover $\gamma$-set $X$ is equivalent
to $X$ $\models$ $S_{1}(B_{\Omega}, B_{\Gamma})$.
\end{proposition}

Note that T. Orenshtein and B. Tsaban proved (Lemma 2.8 in
\cite{orts}) the interesting

\begin{proposition}(Orenshtein, Tsaban) \label{pr11}
Assume that a space $X$ is a Borel-cover $\gamma$-set. Then for
each countable $A\subset B(X)$, each $f\in \overline{A}$ is a
pointwise limit of a sequence of elements of $A$.
\end{proposition}

In \cite{osszts},  authors was proved

\begin{corollary}\label{th121} For a Tychonoff space $X$ the following statements are
equivalent:

\begin{enumerate}

\item $B(X)$ is strongly sequentially separable;

\item $iw(X)=\aleph_0$ and $X$ $\models$ $S_{1}(B_{\Omega},
B_{\Gamma})$;

\item $iw(X)=\aleph_0$ and $X$ is a Borel-cover $\gamma$-set.

\end{enumerate}

\end{corollary}

Recall that the space $Y$ is an $\alpha_1$ space, if for each
$y\in Y$, each countable family $\{A_n\}$ of sequences, each
converging to $y$, can be amalgamated as follows: there are
cofinite subsets $B_n\subset A_n$, $n\in \omega$, such that the
set $B=\bigcup_n B_n$ converges to $y$ \cite{arh0}.

Let $Y$ be a metric space. A function $f: X\mapsto Y$ is a
quasi-normal limit of functions $f_n: X\mapsto Y$ if there are
positive reals $\epsilon_n$, $n\in \omega$, converging to $0$ such
that for each $x\in X$, $d(f_n(x),f(x))<\epsilon_n$ for all but
finitely many $n$. A topological space $X$ is $QN$-space if
whenever $\bf 0$ is a pointwise limit of a sequence of continuous
real-valued functions on $X$, we have that $\bf 0$ is a
quasi-normal limit of the same sequence.

\medskip

By Corollary 20, Corollary 21 and Theorem 32 in \cite{tszd}, we
have

\begin{theorem} For a set of reals $X$ and $0<\alpha\leq \omega_1$, the following are
equivalent.

\begin{enumerate}

\item $C_{p}(X)$ is an $\alpha_1$ space.

\item $B_{\alpha}(X)$ is an $\alpha_i$ space for each
$i=\overline{1,4}$.

\item $X$ is a $QN$ space.

\item $X$ $\models$ $S_1(F_{\Gamma}, F_{\Gamma})$.

\item $X$ $\models$ $S_1(B_{\Gamma}, B_{\Gamma})$.

\item Each Baire class $\alpha$ image of $X$ in $\omega^{\omega}$
is bounded;

\item Each Borel image of $X$ in $\omega^{\omega}$ is bounded.

\item $X$ $\models$
$S_1(Z^{\alpha}_{\Gamma},Z^{\alpha}_{\Gamma})$.

\end{enumerate}
\end{theorem}

Note note that $S_1(B_{\Gamma}, B_{\Gamma})$ implies that $X$ is
$\sigma$-set (Proposition 4 in \cite{scts}). It follows that
$S_1(B_{\Gamma},
B_{\Gamma})=S_1(Z^{\alpha}_{\Gamma},Z^{\alpha}_{\Gamma})$.

\begin{corollary} For a set of reals $X$, the following are
equivalent.

\begin{enumerate}

\item $X$ $\models$ $S_1(Z^{\alpha}_{\Gamma},
Z^{\alpha}_{\Gamma})$.

\item Every $F_{\sigma}$-measurable mapping $f: X \mapsto
\omega^{\omega}$ (i.e., $f^{-1}(U)\in F_{\sigma}$ for open $U$)
the image $f(X)\subset \omega^{\omega}$ is bounded.

\end{enumerate}

\end{corollary}

Recall the definition of the weak distributive law for a family of
subsets of a nonempty set (see \cite{bu}).

Let $X$ be a nonempty set, $\mathcal{A}\subseteq \mathcal{P}(X)$
being a family of subsets. $\mathcal{A}$ is called {\it weakly
distributive} if for any system $A_{n,m}\in \mathcal{A}$, $n,m \in
\omega$ such that

$\bigcap\limits_n \bigcup\limits_m A_{n,m}=X$

there exists a function $\varphi\in \omega^{\omega}$ such that
$\bigcup\limits_{k}\bigcap\limits_{n\geq k} \bigcup\limits_{m\leq
\varphi(n)} A_{n,m}=X$.

\begin{theorem}\label{th20} Let $X$ be a perfectly normal topological space and $0<\alpha\leq \omega_1$. Then the following statements are
equivalent:

\begin{enumerate}

\item $B_{\alpha}(X)$ is strongly sequentially separable;

\item $iw(X)=\aleph_0$ and $X$ $\models$ $S_{1}(B_{\Omega},
B_{\Gamma})$;

\item  $iw(X)=\aleph_0$ and $X$ $\models$ $S_{1}(F_{\Omega},
F_{\Gamma})$.

\end{enumerate}

\end{theorem}

\begin{proof}
Note that if $X$ $\models$ $S_{1}(F_{\Omega}, F_{\Gamma})$ then
the family of closed subsets of $X$ is weakly distributive. Then
$X$ is a $\sigma$-set (see Theorem 5.2. in \cite{bu} and
\cite{kur}).

\end{proof}
\medskip

\begin{corollary}(see Corollary 5.4 in \cite{bu})
Let $X$ be a perfectly normal space. If $X$ $\models$
$S_{1}(F_{\Omega}, F_{\Gamma})$ then every subset of $X$ is a
$QN$-space.
\end{corollary}

\begin{theorem} $(\mathfrak{p}=\mathfrak{c})$ There is a consistent example of a set of reals $X$, such that
$C_p(X)$ is strongly sequentially separable, but $B_1(X)$ is not
strongly sequentially separable.
\end{theorem}

\begin{proof} By corollary 1 in \cite{re} (Rec\l aw's proof assumes Martin's axiom, but the partial order used is $\sigma$-centered
 so that in fact $\mathfrak{p}=\mathfrak{c}$ is enough (see Theorem 44 in \cite{scts})), there exists a $\gamma$-set
which can be mapped onto $[0,1]$ by a Borel function. It follows
that $X$ $\models$ $S_{1}(\Omega, \Gamma)$, but $X$ has not the
property $S_{1}(B_{\Omega}, B_{\Gamma})$. By Theorem \ref{th100}
and Theorem \ref{th20}, $C_p(X)$ is strongly sequentially
separable, but $B_1(X)$ is not strongly sequentially separable.
\end{proof}

\begin{corollary}($\mathfrak{p}=\mathfrak{c}$) There is a consistent example of a set of reals $X$, such
that $X$ $\models$ $S_{1}(\Omega, \Gamma)$, but $X$ has not the
property $ S_{1}(F_{\Omega}, F_{\Gamma})$.
\end{corollary}

For an uncountable cardinal number $\kappa$ a set of real numbers
is a $\kappa$-Sierpi$\acute{n}$ski set if it has cardinality at
least $\kappa$, but its intersection with each set of Lebesgue
measure zero is less than $\kappa$. Note that every
$\mathfrak{b}$-Sierpi$\acute{n}$ski set has property
$S_1(B_{\Gamma},B_{\Gamma})$ and, hence, it is a $\sigma$-set.
Since sets of real numbers having property
$S_1(\mathcal{O},\mathcal{O})$ have measure zero, no
$\mathfrak{b}$-Sierpi$\acute{n}$ski set has property
$S_1(\mathcal{O},\mathcal{O})$. Hence,
$\mathfrak{b}$-Sierpi$\acute{n}$ski set has not property
$S_1(B_{\Omega},B_{\Gamma})$ (\cite{scts}).

By Corollary \ref{th92} and Theorem \ref{th20}, we have the next

\begin{proposition} Let $X$ be a $\mathfrak{b}$-Sierpi$\acute{n}$ski set. Then $B(X)$ is sequentially
separable, but is not strongly sequentially separable.
\end{proposition}

\section{Open questions}

{\bf Question 1.} Assume that for a Tychonoff space $X$, $C_p(X)$
is strongly sequentially separable and $\tau< \mathfrak{p}$. Does
it follow that $C_p(X, \mathbb{R}^{\tau})$ is strongly
sequentially separable ?

\medskip

Recall that a separable space $X$ is said to be a CDH (countable
dense homogeneous) space, if for any two countable dense subsets
$A$ and $B$ in $X$, there is an autohomeomorphism $h$ of $X$ such
that $h(A)=B$.

Note that $\mathbb{R}^{\kappa}$ is a $CDH$ space iff $\kappa<
\mathfrak{p}$.

\medskip

Clearly, that a sequentially separable $CDH$ space is a strongly
sequentially separable.

\medskip

{\bf Question 2.} Assume that for a Tychonoff space $X$, $C_p(X)$
is strongly sequentially separable and $\tau< \mathfrak{p}$. Does
it follow that $C_p(X, \mathbb{R}^{\tau})$ is a $CDH$ space ?

\medskip
{\bf Question 3.} Assume that for Tychonoff space $X$, $B(X)$ is
strongly sequentially separable and $\tau< \mathfrak{p}$. Does it
follow that $B(X, \mathbb{R}^{\tau})$ is a $CDH$ space ?

\medskip

{\bf Question 4.} Assume that for a set of reals $X$, $B(X)$ is
strongly sequentially separable. Does it follow that $B(X^n)$ is
strongly sequentially separable for any $n>1$ ?

\medskip

{\bf Question 5.} Assume that a space $X$ is a $\sigma$-,
$\gamma$-set. Does it follow that $X^n$ is a $\sigma$-,
$\gamma$-set for any $n>1$ ?

\medskip

{\bf Question 6.} Assume that there is a Baire isomorphism (class
$\alpha$) from a $\sigma$-set $X$ onto a Tychonoff space $Y$. Does
it follow that $Y$ is a $\sigma$-space ?

\medskip

{\bf Question 7.} Does {\bf MA} imply the existence of a
$\sigma$-, $\gamma$-set of size the continuum ?

\medskip

{\bf Question 8.} (MA) Assume that there is a $\sigma$-set $X$
such that $X$ is a $\gamma$-set and $|X|=\mathfrak{c}$. Does it
follow that for any $Y\subset X$, $Y$ is a $\gamma$-set ?

\medskip

{\bf Question 9.} Assume that $B_{\alpha}(X)$ is not sequentially
separable. Does it follow that $B_{\beta}(X)$ is not sequentially
separable for $\beta>\alpha$?

\medskip

{\bf Acknowledgment.} The author wishes to express his Thanks to
Boaz Tsaban for his inspiration for writing this paper and for the
detailed answers on the numerous author's questions.

\bibliographystyle{model1a-num-names}
\bibliography{<your-bib-database>}







\end{document}